\documentclass[12pt]{article}
\usepackage{graphicx}
\usepackage{amsmath,amsthm,amssymb,enumerate}
\usepackage{euscript,mathrsfs}
\usepackage[left=2cm,right=2cm,top=3.5cm,bottom=3.5cm]{geometry}
\usepackage{color}

\usepackage{soul}

\catcode`\@=11 \@addtoreset{equation}{section}

\catcode`\@=12

\allowdisplaybreaks

\newtheorem{Theorem}{Theorem}[section]
\newtheorem{Proposition}[Theorem]{Proposition}
\newtheorem{Lemma}[Theorem]{Lemma}
\newtheorem{Corollary}[Theorem]{Corollary}

\theoremstyle{definition}
\newtheorem{Definition}[Theorem]{Definition}

\newtheorem{Remark}[Theorem]{Remark}

\newcommand{\bTheorem}[1]{
\begin{Theorem} \label{T#1} }
\newcommand{\eT}{\end{Theorem}}

\newcommand{\bProposition}[1]{
\begin{Proposition} \label{P#1}}
\newcommand{\eP}{\end{Proposition}}

\newcommand{\bLemma}[1]{
\begin{Lemma} \label{L#1} }
\newcommand{\eL}{\end{Lemma}}

\newcommand{\bCorollary}[1]{
\begin{Corollary} \label{C#1} }
\newcommand{\eC}{\end{Corollary}}

\newcommand{\bRemark}[1]{
\begin{Remark} \label{R#1} }
\newcommand{\eR}{\end{Remark}}

\newcommand{\bDefinition}[1]{
\begin{Definition} \label{D#1} }
\newcommand{\eD}{\end{Definition}}

\newcommand{\vrB}{\vr_B}

\newcommand{\Ds}{\mathbb{D}_x}

\newcommand{\vuB}{\vc{u}_B}

\newcommand{\bfphi}{\boldsymbol{\varphi}}

\newcommand{\bFormula}[1]{
\begin{equation} \label{#1}}
\newcommand{\eF}{\end{equation}}

\newcommand{\Ov}[1]{\overline{#1}}

\newcommand{\DC}{C^\infty_c}

\newcommand{\aleq}{\stackrel{<}{\sim}}

\newcommand{\ageq}{\stackrel{>}{\sim}}

\newcommand{\vr}{\varrho}

\newcommand{\tvr}{\tilde \vr}
\newcommand{\tvu}{{\tilde \vu}}

\newcommand{\vu}{\vc{u}}
\newcommand{\vm}{\vc{m}}

\newcommand{\vc}[1]{{\bf #1}}

\newcommand{\Div}{{\rm div}_x}
\newcommand{\Grad}{\nabla_x}

\newcommand{\dx}{\,{\rm d} {x}}

\newcommand{\dt}{\,{\rm d} t }

\newcommand{\intO}[1]{\int_{\Omega} #1 \ \dx}

\newcommand{\D}{{\rm d}}

\newcommand{\ep}{\varepsilon}

\def\softd{{\leavevmode\setbox1=\hbox{d}%
          \hbox to 1.05\wd1{d\kern-0.4ex{\char039}\hss}}}
\definecolor{Cgrey}{rgb}{0.85,0.85,0.85}
\definecolor{Cblue}{rgb}{0.50,0.85,0.85}
\definecolor{Cred}{rgb}{1,0,0}
\definecolor{fancy}{rgb}{0.10,0.85,0.10}

\newcommand\Cbox[2]{%
    \newbox\contentbox%
    \newbox\bkgdbox%
    \setbox\contentbox\hbox to \hsize{%
        \vtop{
            \kern\columnsep
            \hbox to \hsize{%
                \kern\columnsep%
                \advance\hsize by -2\columnsep%
                \setlength{\textwidth}{\hsize}%
                \vbox{
                    \parskip=\baselineskip
                    \parindent=0bp
                    #2
                }%
                \kern\columnsep%
            }%
            \kern\columnsep%
        }%
    }%
    \setbox\bkgdbox\vbox{
        \color{#1}
        \hrule width  \wd\contentbox %
               height \ht\contentbox %
               depth  \dp\contentbox
        \color{black}
    }%
    \wd\bkgdbox=0bp%
    \vbox{\hbox to \hsize{\box\bkgdbox\box\contentbox}}%
    \vskip\baselineskip%
}

\date{}

\begin{document}


\title{Weak--strong uniqueness property for models of compressible viscous fluids
near vacuum}

\author{Eduard Feireisl
\thanks{The work of E.F. was supported by the
Czech Sciences Foundation (GA\v CR), Grant Agreement
21--02411S.}
\and Anton\' \i n Novotn\' y { \thanks{The work of A.N. was partially supported by the Eduard \v Cech visiting program at the Mathematical Institute of the Academy of Sciences of the Czech Republic. }
}
}


\maketitle

\centerline{Institute of Mathematics of the Academy of Sciences of the Czech Republic;}
\centerline{\v Zitn\' a 25, CZ-115 67 Praha 1, Czech Republic}

\centerline{feireisl@math.cas.cz}

\centerline{and}

\centerline{IMATH, EA 2134, Universit\'e de Toulon,}
\centerline{BP 20132, 83957 La Garde, France}
\centerline{novotny@univ-tln.fr}

\begin{abstract}

We extend the weak--strong uniqueness principle to general models of compressible viscous fluids near/on the vacuum. In particular, the physically relevant case of 
positive density with polynomial decay at infinity is considered.

\end{abstract}


\bigskip


\section{Introduction}
\label{I}

The motion of a compressible viscous fluid in the barotropic regime is described by the time evolution of the mass density
$\vr = \vr(t,x)$ and the velocity field $\vu = \vu(t,x)$, $t > 0$, $x \in \Omega \subset R^d$, $d=1,2,3$, satisfying the equation of continuity
\begin{equation} \label{I1}
\partial_t \vr + \Div (\vr \vu) = 0,
\end{equation}
and the momentum balance
\begin{equation} \label{I2}
\partial_t (\vr \vu) + \Div (\vr \vu \otimes \vu) + \Grad p(\vr) = \Div \mathbb{S} + \vr \vc{f}.
\end{equation}
Here, $p = p(\vr)$ is the pressure, $\vc{f}$ the external driving force, and $\mathbb{S}$ the viscous stress. To close the system,
we suppose that $\mathbb{S}$ is related to the symmetric part of the velocity gradient
\[
\Ds \vu \equiv \frac{1}{2} \left( \Grad \vu + \Grad^t \vu \right)
\]
through a general rheological law
\begin{equation} \label{I3}
F(\Ds \vu) + F^* (\mathbb{S}) = \mathbb{S} : \Grad \vu \ \Leftrightarrow \
\mathbb{S} \in \partial F (\Ds \vu) \ \Leftrightarrow \
\Ds \vu \in \partial F^* (\mathbb{S}),
\end{equation}
where
\[
F: R^{d \times d}_{\rm sym} \to [0, \infty] \ \mbox{is a convex l.s.c. function.}
\]
The best known example is the isentropic Navier--Stokes system, where
\begin{equation} \label{I4}
p(\vr) = a \vr^{\gamma},\ a > 0, \ \gamma > 1,\
F = \frac{\mu}{2} \left| \Grad \vu + \Grad^t \vu - \frac{2}{d} \Div \vu \mathbb{I} \right|^2 +
\frac{\lambda}{2} |\Div \vu|^2, \mu > 0, \ \lambda \geq 0.
\end{equation}
The ``implicit'' rheological law \eqref{I3} covers the class of power law fluids as well as other non-Newtonian fluids, 
see e.g. Bul\' \i \v cek et al. \cite{BuGwMaSG}.

Our goal is to study stability of strong solutions of the problem in the regime when the fluid density either vanishes or approaches
asymptotically the vacuum state. To begin, we point out that the model has been derived for non--dilute fluids out of vacuum. In particular, as strong solutions satisfy the equation of continuity
\[
\partial_t \vr + \vu \cdot \Grad \vu + \vr \Div \vu = 0,
\]
the density will always remain positive unless the vacuum state is artificially imposed through the initial or boundary data. However, allowing
$\vr \approx 0$ is still physically relevant at least in certain asymptotic regimes:
\begin{itemize}
\item Solutions of \eqref{I1}--\eqref{I3} with $\vc{f} = \Grad G$, $G = G(x)$ approach an equilibrium state $\vr = \tvr$,
$\vu = 0$ as $t \to \infty$, see \cite{FeiKwoNov}, where
\begin{equation} \label{I5}
\Grad p(\tvr) = \tvr \Grad G.
\end{equation}
It is a simple observation that $\tvr$ may vanish on a non--void subset of the physical domain, in particular if the total mass of the fluid is small enough.

\item In models of gaseous stars in astrophysics, the physical domain $\Omega$ is exterior to a rigid object, while
\begin{equation} \label{I5a}
\vr \to 0,\ \vu \to 0 \ \mbox{as}\ |x| \to \infty.
\end{equation}
Thus the density is close to zero at least in the far field. Note that, on the one hand, the isentropic Navier--Stokes system \eqref{I1}, \eqref{I2}, \eqref{I4} is globally well posed in this regime if the total energy is small enough, see Huang, Li, and Xin
\cite{HuLiXi},
and also Fang, Zhu, and Guo \cite{FanZhuGuo} for a similar result in the non-Newtonian case and $d=1$. On the other hand, Rozanova \cite{Roza} and Xin \cite{XIN} showed that fast decay of solutions as $|x| \to \infty$ is not compatible with global existence.
More recently, Merle et al. \cite{MeRaRoSz} obtained blow--up results in a similar regime for radially symmetric solutions with
certain profiles and certain values of the adiabatic exponent $\gamma$.

\item Last but not least, the problem with vanishing initial density is mathematically challenging and has been considered
in a large number of recent studies, see Gong et al. \cite{GoLiLiZh}, Li and Xin \cite{LiXin}, Liang \cite{Liang},
to name only a few.

\end{itemize}

Our aim is to clarify in which way the strong solutions obtained in the above references can coexist/coincide with the weak solutions
in the sense of P.-L.Lions \cite{LI4} or even more general dissipative solutions introduced in \cite{AbbFeiNov}. More specifically, we establish the weak--strong strong uniqueness principle: A weak solution of the problem \eqref{I1}--\eqref{I3} coincides with the
strong one with the same initial/boundary data as long as the strong solution exists. As shown in \cite[Theorem 6.3]{AbbFeiNov}, the
weak--strong uniqueness principle holds for the problem \eqref{I1}--\eqref{I3} in the class of strong solutions away from vacuum, and
for $\Omega \subset R^d$ a bounded domain with general in/out flow boundary conditions:
\begin{equation} \label{I6}
\vu|_{\partial \Omega} = \vuB,\ \vr|_{\Gamma_{\rm in}} = \vrB,\
\Gamma_{\rm in} \equiv \left\{ x \in \partial \Omega \ \Big|\ \vuB \cdot \vc{n} < 0 \right\},
\end{equation}
where $\vc{n}$ is the outer normal vector to $\partial \Omega$. We show that this result can be extended to the class of strong solutions
near/with the vacuum as long as:
\begin{enumerate}
\item
The potential $F$ is uniformly strictly convex in $R^{d \times d}_{{\rm sym},0}$.
\item
The pressure $p$ is strictly convex in $[0, \infty)_{\rm loc}$.
\item
The density component $\tvr$ of the strong solution satisfies
\[
\Grad P'(\tvr) \in L^2(0,T; L^q(\Omega)),\ q= q(d),
\]
where $P$ is the pressure potential,
\[
P'(\vr) \vr - P(\vr) = p(\vr).
\]
\item
If $\Omega$ is unbouded, the initial density must decay sufficiently fast as $|x| \to \infty$.

\end{enumerate}

Conditions 1, 2 are constitutive. Condition 1 holds in particular for the Newtonian stress in the Navier--Stokes system.
Condition 2 is satisfied for $p = a \vr^\gamma$ as long as $1 < \gamma \leq 2$, which is physically relevant in the isentropic regime.
Condition 3 is associated to the flow. If the velocity field $\tvu$ is smooth enough, the Sobolev regularity of
$P'(\tvr)$ is propagated along streamlines. It is therefore enough to impose 3 on the data only.
We point out that 
\[
\Grad P'(\tvr) \approx \tvr^{\gamma - 2} \Grad \tvr
\]
becomes singular near the vacuum if $\gamma < 2$. Condition 4 is relevant if the physical domain is unbounded. Here again, the decay may be inherited from the initial/boundary data if the velocity field is sufficiently regular.

Our method is based on a general relative energy inequality obtained in \cite{AbbFeiNov} that must be adapted to the vanishing density regime. The most delicate issue is the behavior of the fluid velocity on or near the vacuum region. As a matter of fact, the concept of velocity
in the absence of matter is meaningless, however, the velocity can be recovered as a solution of an elliptic equation in the whole physical space. This paradox can be seen as a consequence of the infinite speed of propagation due to the viscous stress with constant
viscosity coefficients. {By the same token, the decay of the velocity for $|x| \to \infty$ is not expected to be faster than that of the Dirichlet kernel at least in the linearly viscous case. This fact makes the analysis on unbounded domains rather delicate and requires certain decay properties of the initial density profile. Seen from this perspective, the hypothesis $\vu (t, \cdot) \in W^{1,2}$ assumed sometimes in the literature is not very realistic.}

The paper is organized as follows. We begin by recalling the basic concepts of strong as well as weak (dissipative) solution in Section
\ref{sw}. In Section \ref{re}, we introduce the relative energy functional. The main results are stated and proved in Section \ref{ws}.
We distinguish three cases: Bounded domain, general unbounded domain, and the whole space $R^3$ with 
compactly supported density or rapidly decaying density.

\section{Strong and weak solutions}
\label{sw}

We recall the definition of strong and weak (dissipative) solutions to the problem \eqref{I1}--\eqref{I3}, with the boundary
conditions \eqref{I6}, and with the far field conditions \eqref{I5a} if the domain $\Omega$ is unbounded. For the sake of simplicity, 
we assume that the boundary data $\vuB$, $\vrB$ are independent of $t$.
First, extend the
function $\vuB$ to be smooth in $\Ov{\Omega}$ and
\begin{equation} \label{sw1}
\vu_B(x) = 0 \ \mbox{for}\ |x| > R
\end{equation}
if $\Omega$ is unbounded. To avoid technical difficulties, we suppose that the outer normal is defined whenever $\vuB \cdot \vc{n}
\ne 0$. Moreover, we consider the pressure in the isentropic form
\begin{equation} \label{sw1a}
p(\vr) = a \vr^\gamma, \ a > 0, \ \gamma > 1.
\end{equation}
Finally, we impose certain growth restrictions on the dissipation potential $F$,
\begin{equation} \label{ws1a}
0 \leq F(\mathbb{D}) \aleq |\mathbb{D}|^2,
\end{equation}
\begin{equation} \label{ws1}
F( \mathbb{D} + \mathbb{Q} ) - F(\mathbb{D}) - \mathbb{S} : \mathbb{Q} \ageq
\left|\mathbb{Q} - \beta {\rm tr}[\mathbb{Q}] \mathbb{I} \right|^2,\ \beta = \frac{1}{d} \ \mbox{if}\ d = 2,3, \ \beta = 0
\ \mbox{if}\ d = 1,
\end{equation}
for any $\mathbb{D}, \mathbb{Q} \in R^{d \times d}_{\rm sym}$, $\mathbb{S} \in \partial F (\mathbb{D})$. Note that both \eqref{ws1a}
and \eqref{ws1} are compatible with the Newtonian potential \eqref{I4}. The condition \eqref{ws1a} can be relaxed at the expense
of several technical difficulties in the proofs. The coercivity assumption \eqref{ws1}, reflecting non-degeneracy of the viscous stress,
is essential.

\subsection{Strong solutions}
\label{Ssw1}

The goal is to identify the largest possible class of strong solutions. Inspired by the existence result by Cho, Choe, Kim \cite{ChoChoeKim} we consider the velocity $\tvu$ in the class
\begin{equation} \label{sw2}
(\tvu - \vuB) \in C([0,T]; D^{1,2}_0 (\Omega; R^d) \cap L^2(0,T; D^{2, q}(\Omega; R^d)),
\partial_t \tvu \in L^2(0,T; L^{q} (\Omega; R^d))
\end{equation}
for some $q > d$.
Here, $D^{k,q}$ denotes the space of locally integrable functions $v$ with $D^k_x v \in L^q(\Omega)$, $D^{1,2}_0$ is the closure
of $\DC(\Omega)$ in $D^{1,2}(\Omega)$. If $\Omega$ is bounded, $D^{k,p}$ coincide with the standard Sobolev spaces $W^{k,p}$.
Unbounded domains will be considered only for $d=3$.
Note that the class \eqref{sw2} includes the case $q = 6$, $d=3$ considered in \cite{ChoChoeKim}. In particular,
by interpolation and the standard embedding $D^{1, q} \hookrightarrow C$, $q > d$,
\[
\tvu \in C([0,T]; D^{1,q} \cap D^{1,2}_0 (\Omega; R^d)) \hookrightarrow BC([0,T] \times \Ov{\Omega}; R^d),
\]
{ where  $BC(Q)$ is the Banach space of bounded continuous functions on $Q$.}
Similarly, we deduce from \eqref{sw2}
\begin{equation} \label{sw3}
\int_0^T  \| \tvu \|_{W^{1,\infty}(\Omega; R^d)}^2 \dt \aleq 1
\end{equation}
if either $\Omega$ is bounded or $d=3$.

In particular, given the velocity field $\tvu$, the initial state and the boundary data
\[
\tvr(0, \cdot) = \vr_0 \in L^1_{\rm loc}(\Omega),\ \tvr|_{\Gamma_{\rm in}} = \vrB \in L^1(\partial \Omega),
\]
the density $\tvr$ is uniquely determined by the equation of continuity \eqref{I1} and can be computed explicitly via the method of
characteristics. More specifically, supposing $\Omega$ has a compact Lipschitz boundary, we may extend $\tvu$ outside $\Omega$
to remain in the class
\[
\tvu \in L^2(0,T; D^{2, q}(R^d; R^d)) \cap BC([0,T] \times R^d; R^d).
\]
Then we define the characteristic curves
\[
\frac{\D }{\dt} \vc{X}(t,x) = \tvu (t, \vc{X}(t,x)),\ \vc{X}(0, x) = x.
\]
For any $t > 0$, $x \in \Omega$, there is a unique characteristic satisfying
\[
\vc{X}(t, x_0) = x.
\]
Let $\tau = \inf \left\{ s \in [0,t] \ \Big| \vc{X}(z, x_0) \in \Omega \ \mbox{for all} \ z \in (s,t) \right\}$.
If $\tau = 0$, we set
\begin{equation} \label{sw4}
\tvr(t,x) = \vr_0(x_0) \exp \left( - \int_0^t \Div \tvu (s, \vc{X}(s,x_0) ) \D s \right).
\end{equation}
If $\tau > 0$, then necessarily
\[
x_\tau = \vc{X}(\tau, x_0) \in \Ov{\Gamma}_{\rm in},
\]
and we set
\begin{equation} \label{sw5}
\tvr(t,x) = \vrB(x_\tau) \exp \left( - \int_\tau^t \Div \tvu (s, \vc{X}(s,x_\tau) ) \D s \right).
\end{equation}
{To avoid jumps in the resulting density, the above procedure requires an obvious compatibility condition} 
\[
\vr_0 \in BC(\Ov{\Omega}), \ \vr_0|_{\Gamma_{\rm in}} = \vrB.
\]
{Still the construction via characteristic lines may give rise to ambiguous results as soon as part of the boundary is tangent to the characteristic 
curves. To avoid this problem, several non--degeneracy conditions on $\Gamma_{\rm in}$ are usually imposed in the literature
to eliminate ingoing characteristics emanating from $\partial \Gamma_{\rm in}$, cf. e.g. Valli and Zajaczkowski \cite{VAZA}.}
The density $\tvr \in BC([0,T] \times \Ov{\Omega})$ associated to $\tvu$ is then well defined through formulae
\eqref{sw4}, \eqref{sw5}.
As a direct consequence,
we obtain the following standard result concerning propagation of Sobolev regularity. 
{To avoid the afore mentioned technical problems and also because
the available regularity of
$\partial_t \tvu$ is rather limited, we restrict ourselves to the case $\Gamma_{\rm in} = \emptyset$.}

\begin{Lemma} \label{Lsw1}
Let $\Omega \subset R^d$ be a domain with compact (possibly empty) Lipschitz boundary. Let a vector field $\tvu$ be given in the class
\[
\tvu \in L^2(0,T; D^{2, q}(\Omega; R^d)) \cap BC([0,T] \times \Ov{\Omega}; R^d),\ d < q \leq \infty
\]
and such that $\tvu \cdot \vc{n}|_{\partial \Omega} \geq 0$. Suppose that
\[
\vr_0 \in BC(\Ov{\Omega}),\
\vr_0 \geq 0,  \ \Grad (\vr_0^{\gamma-1}) \in L^q (\Omega; R^d) \ \mbox{for some}\ \ 1 < \gamma \leq 2.
\]

Then the unique solution $\tvr$ of the equation of continuity \eqref{I1} determined by \eqref{sw4} satisfies
\[
\| \Grad (\tvr^{\gamma - 1}) (t, \cdot) \|_{L^q(\Omega; R^d) } \leq c (T, \vr_0, \tvu) \ \mbox{for any}\ t \in [0,T].
\]

\end{Lemma}

\begin{proof}

Extending $\tvu$ and $\vr_0$ outside $\Omega$ as the case may be, we may assume $\Omega = R^d$. As $\Gamma_{\rm in} = \emptyset$, the solution $\tvr$ is given by formula \eqref{sw4}. In particular,
\[
\tvr (t,x) = \vr_0 (\vc{X}^{-1}(t,x) )\exp \left( - \int_0^t \Div \tvu (s, \vc{X}(s,\vc{X}^{-1}(t,x) ) ) \D s \right),
\]
\[
\tvr^{\gamma - 1} (t,x) = \vr^{\gamma - 1}_0 (\vc{X}^{-1}(t,x) )\exp \left( (1 - \gamma) \int_0^t \Div \tvu (s, \vc{X}(s,\vc{X}^{-1}(t,x) ) ) \D s \right).
\]
Thus the desired result follows by direct  differentiation of the above formula.

\end{proof}

If the exponent $\gamma$ remains in the range $1 < \gamma \leq 2$, differentiablity of $\tvr^{\gamma -1}$ implies differentiablity
of $\tvr$ and of $\tvr^\gamma$. In particular, the equation of continuity \eqref{I1} can be interpreted in the strong sense. Similarly, the left hand side of the momentum equation can be written in the form
\[
\tvr \partial_t \tvu + \tvr \tvu \cdot \Grad \tvu + \tvr \Grad \tvr^{\gamma - 1}
\]
Thus for $\tvu$ belonging to the class \eqref{sw2}, the system \eqref{I1}--\eqref{I3} can be written as
\begin{equation} \label{sw7}
\begin{split}
\partial_t \tvr + \tvu \cdot \Grad \tvr + \tvr \Div \tvu &= 0,\\
\tvr \partial_t \tvu + \tvr \tvu \cdot \Grad \tvu + \tvr \Grad \tvr^{\gamma - 1} &= \Div \widetilde{\mathbb{S}} + \tvr \vc{f},
\end{split}
\end{equation}
with
\begin{equation} \label{sw8}
\widetilde{\mathbb{S}} \in \partial F(\Ds \tvu) \ \mbox{for a.a.} \ t \in (0,T).
\end{equation}

\subsection{Dissipative (weak) solutions}

The dissipative solutions to the system \eqref{I1}--\eqref{I3} were introduced in \cite{AbbFeiNov}. We denote
$\mathcal{M}^+(\Ov{\Omega}; R^{d \times d}_{\rm sym})$ the set of positively semi--definite finite tensorial measures on $\Ov{\Omega}$.
Specifically, $\mathfrak{R} \in \mathcal{M}^+(\Ov{\Omega}; R^{d \times d}_{\rm sym})$ if $\mathfrak{R}$ is a tensor valued finite measure satisfying
\[
\mathfrak{R}: (\xi \otimes \xi) \geq 0 \ \mbox{for any}\ \xi \in R^d.
\]

\begin{Definition}[{\bf Dissipative solution}] \label{Dsw1}

The functions $(\vr, \vu)$ represent \emph{dissipative solution} of the problem \eqref{I1}--\eqref{I3}, with the boundary conditions \eqref{I6}, the far field conditions \eqref{I5a}, and the initial conditions
\[
\vr(0, \cdot) = \vr_0, \ (\vr \vu)(0, \cdot) = \vc{m}_0,
\]
if the following holds:
\begin{itemize}
\item {\bf Integrability.} 
\[
\begin{split}
\vr \geq 0,\ \vr &\in C_{\rm weak}([0,T]; L^\gamma(\Omega)) \cap
L^\gamma((0,T); L^\gamma (\partial \Omega, |\vuB \cdot \vc{n} | \D \sigma)),\\
(\vu - \vuB) &\in L^2(0,T; D^{1,2}_0(\Omega; R^d)),\ \vr \vu \in C_{\rm weak}([0,T]; L^{\frac{2 \gamma}{\gamma + 1}}(\Omega; R^d)),\\
\mathbb{S} &\in L^2 (0,T; L^2(\Omega; R^{d \times d}_{\rm sym})).
\end{split}
\]
\item {\bf Equation of continuity.}
The integral identity
\begin{equation} \label{sw9}
\begin{split}
\left[ \intO{ \vr \varphi } \right]_{t = 0}^{t = \tau} &+
\int_0^\tau \int_{\partial \Omega} \varphi \vr [\vu_B \cdot \vc{n}]^+ \ \D \ S_x
+
\int_0^\tau \int_{\partial \Omega} \varphi \vr_B [\vu_B \cdot \vc{n}]^- \ \D \ S_x\\ &=
\int_0^\tau \intO{ \Big[ \vr \partial_t \varphi + \vr \vu \cdot \Grad \varphi \Big] } \dt ,\ \vr(0, \cdot) = \vr_0
\end{split}
\end{equation}
holds
for any $0 \leq \tau \leq T$, and any test function $\varphi \in C^1_c([0,T] \times \Ov{\Omega})$.

\item {\bf Momentum equation.}
There exists
\[
\mathfrak{R} \in L^\infty_{\rm weak(*)}(0,T; \mathcal{M}^+ (\Ov{\Omega}; R^{d \times d}_{\rm sym}))
\]
such that
\begin{equation} \label{sw10}
\begin{split}
\left[ \intO{ \vr \vu \cdot \bfphi } \right]_{t=0}^{t = \tau} &=
\int_0^\tau \intO{ \Big[ \vr \vu \cdot \partial_t \bfphi + \vr \vu \otimes \vu : \Grad \bfphi
+ p(\vr) \Div \bfphi - \mathbb{S} : \Grad \bfphi  + \vr \vc{f} \cdot \bfphi \Big] }\\
&+ \int_0^\tau \int_{{\Omega}} \Grad \bfphi : \D \ \mathfrak{R}(t) \ \dt,\ \vr \vu(0, \cdot) = \vm_0
\end{split}
\end{equation}
holds for any $0 \leq \tau \leq T$ and any test function $\bfphi \in C^1_c([0,T] \times \Ov{\Omega}; R^d)$, $\bfphi|_{\partial \Omega} = 0$.

\item {\bf Energy inequality.}
There exists
\[
\mathfrak{E} \in L^\infty_{\rm weak(*)}(0,T; \mathcal{M}^+ (\Ov{\Omega}))
\]
such that
\begin{equation} \label{sw11}
\begin{split}
&\left[ \intO{\left[ \frac{1}{2} \vr |\vu - \vuB|^2 + P(\vr) \right] } \right]_{t = 0}^{ t = \tau} +
\int_0^\tau \intO{ \Big[ F(\Ds \vu) + F^* (\mathbb{S}) \Big] } \dt\\
&+\int_0^\tau \int_{\partial \Omega} P(\vr)  [\vu_B \cdot \vc{n}]^+ \ \D S_x \dt +
\int_0^\tau \int_{\partial \Omega} P(\vr_B)  [\vu_B \cdot \vc{n}]^- \ \D S_x \dt
+ \int_{\Ov{\Omega}}  \D \ \mathfrak{E} (\tau) \\	
\leq
&-
\int_0^\tau \intO{ \left[ \vr \vu \otimes \vu + p(\vr) \mathbb{I} \right]  :  \Grad \vu_B } \dt + \int_0^\tau { \frac{1}{2} \intO{ {\vr} \vu\cdot \Grad |\vu_B|^2  } }
\dt\\& + \int_0^\tau \intO{ \Big[ \mathbb{S} : \Grad \vu_B + \vr \vc{f} \cdot { (\vu-\vu_B)} \Big] } \dt\\
&-\int_0^\tau \int_{\Ov{\Omega}} \Grad \vu_B : \D \ \mathfrak{R}(t) \dt,\ P(\vr) \equiv \frac{a}{\gamma - 1} \vr^\gamma
\end{split}
\end{equation}
for a.a. $0 \leq \tau \leq T$.

\item {\bf Defect compatibility.}

\begin{equation} \label{sw12}
\underline{d} \mathfrak{E} \leq {\rm tr}[\mathfrak{R}] \leq \Ov{d} \mathfrak{E},\
\ \mbox{for certain constants}\ 0 < \underline{d} \leq \Ov{d}.
\end{equation}

\end{itemize}

\end{Definition}
\begin{Remark} \label{Rsw1}

The hypothesis
\[
\mathbb{S} \in L^2(0,T; L^2(\Omega; R^{d \times d}_{\rm sym}))
\]
is pertinent to the class of the dissipative potentials $F$ with at most quadratic growth \eqref{ws1a}.

\end{Remark}
\section{Relative energy}
\label{re}

The relative energy reads:
\[
E \left( \vr, \vu \ \Big|\ \tvr, \tvu \right) =
\frac{1}{2} \vr |\vu - \tvu|^2 + P(\vr) - P'(\tvr) (\vr - \tvr) - P(\tvr).
\]
The relative energy inequality associated to the problem \eqref{I1}--\eqref{I3}, \eqref{I6} was derived in
\cite[Section 5, formula (5.5)]{AbbFeiNov}. Its generalization to the case of unbounded domain is straightforward { (cf. also \cite{FeJiNo})}:
\begin{equation} \label{re1}
\begin{split}
&\left[ \intO{E\left( \vr, \vu \ \Big|\ \tvr, \tvu \right) } \right]_{t = 0}^{ t = \tau} +
\int_0^\tau \intO{ \Big[ F(\Ds \vu) + F^* (\mathbb{S}) \Big] } \dt - \int_0^\tau \intO{ \mathbb{S} : \Grad \tvu } \dt \\
&{ +\int_0^\tau \int_{\partial \Omega} \left[ P(\vr) - P'(\tvr) (\vr - \tvr) - P(\tvr)  \right]  [\vu_B \cdot \vc{n}]^+ \ \D S_x \dt}\\
&{ +\int_0^\tau \int_{\partial \Omega} \left[ P(\vr_B) - P'(\tvr) (\vr_B - \tvr) - P(\tvr)  \right]  [\vu_B \cdot \vc{n}]^- \ \D S_x \dt}
+ \int_{\Ov{\Omega}} 1 \D \ \mathfrak{E} (\tau)\\
\leq &- \int_0^\tau \intO{ \vr (\tvu - \vu) \cdot (\tvu - \vu) \cdot \Grad \tvu } \dt\\
&-
\int_0^\tau \intO{ \Big[ p(\vr) - p'(\tvr) (\vr - \tvr) - p(\tvr) \Big] \Div \tvu } \dt
\\
&+ \int_0^\tau \intO{ \vr (\tvu - \vu) \cdot \Big[ \partial_t \tvu  +  \tvu \cdot \Grad \tvu
  + \Grad P'(\tvr) \Big] } \dt { -\int_0^\tau \intO{ \vr (\tvu - \vu) \cdot \vc{f} } \dt}\\ &+ \int_0^\tau \intO{
	p'(\tvr)\left( 1 - \frac{\vr}{\tvr} \right)  \Big[ \partial_t \tvr   +  \Div (\tvr \tvu)
  \Big] } \dt 
	\\
&- \int_0^\tau \int_{\Ov{\Omega}} \Grad \tvu : \D \ \mathfrak{R}(t) \ \dt
\ \mbox{for a.a.}\ \tau \in (0,T),
\end{split}
\end{equation}
for a.a. $0 < \tau < T$,
and any pair of ``test functions''
\begin{equation} \label{re2}
\begin{split}
\tvu &\in C^1_c([0,T] \times \Ov{\Omega}; R^d), \tvu|_{\partial \Omega} = \vuB,\\
\tvr &\in C^1([0,T] \times \Ov{\Omega}), \ \tvr > 0, (\tvr - \ep) \in C_c([0,T] \times \Ov{\Omega})
\ \mbox{for some}\ \ep > 0.
\end{split}
\end{equation}

Our goal is to extend validity of \eqref{re1} to the class of test functions containing $(\tvr + \ep, \tvu)$, where
$(\tvr, \tvu)$ is a strong solution. Using the standard regularization by a convolution kernel in time, we first observe that
the class of test functions can be extended to
\begin{equation} \label{re3}
\begin{split}
\tvu &\in C_c([0,T] \times \Ov{\Omega}; R^d), \tvu|_{\partial \Omega} = \vuB,\
\partial_t \tvu \in L^2(0,T; C_c(\Ov{\Omega})),\
\Grad \tvu \in L^2(0,T; C_c(\Ov{\Omega}; R^{d \times d})) \\
(\tvr - \ep) &\in C_c([0,T] \times \Ov{\Omega}), \ \tvr > 0, \
\partial_t \tvr \in L^{\infty}_{\rm weak-(*)}(0,T; C_c(\Ov{\Omega})),\
\Grad \tvr \in L^{\infty}_{\rm weak-(*)}(0,T; C_c(\Ov{\Omega}; R^d))
\end{split}
\end{equation}
for some $\ep > 0$. 

The class \eqref{re3} covers all interesting cases if $\Omega$ is bounded. If $\Omega$ is an exterior domain or $\Omega = R^d$, integrability for $|x| \to \infty$ is relevant. In this case, we consider $d=3$ only, where we consider the class:
\begin{equation} \label{re4}
\begin{split}
\tvu &\in BC([0,T] \times \Ov{\Omega}; R^3), \tvu|_{\partial \Omega} = \vuB,\
\Grad \tvu \in L^2(0,T; BC \cap L^2 (\Ov{\Omega}; R^{d \times d})),\\
\partial_t \tvu &\in L^2(0,T; BC\cap L^6(\Ov{\Omega} ; R^3 )),
 \\
\tvr &\in BC([0,T] \times \Ov{\Omega}), \ \inf \tvr > 0, \
\Grad \tvr \in L^{\infty}_{\rm weak-(*)}(0,T; BC \cap L^6 (\Ov{\Omega}; R^3))\\
\partial_t \tvr &\in L^{\infty}_{\rm weak-(*)}(0,T; BC \cap L^6 (\Ov{\Omega})).
\end{split}
\end{equation}
{The extension to this class can be justified by the standard cut-off procedure as long as all integrals
in \eqref{re1} are finite.
Extensions to other classes of test functions can be obtained in a similar manner as the case may be.} 

\section{Weak strong uniqueness}
\label{ws}

We are ready to establish the main results.
We distinguish the cases of $\Omega$ bounded, unbounded, or $\Omega = R^3$. For the sake of simplicity, we suppose here and hereafter that 
\[
\vc{f} = 0.
\]

\subsection{Weak--strong uniqueness--bounded domain}
\label{wsbd}

We start with our main result concerning bounded domains, where inflow is absent.

\begin{Theorem} [{\bf Weak--strong uniqueness, bounded domain, no inflow}] \label{wsT1}

Suppose that $\Omega \subset R^d$, $d=1,2,3$ is a bounded Lipschitz domain. Let the boundary velocity $\vuB$ be a twice continuously differentiable function satisfying
\[
\vuB \cdot \vc{n} \geq 0 \ \mbox{on}\ \partial \Omega.
\]
Let $(\tvr, \tvu)$ be a strong solution of the problem \eqref{I1}, \eqref{I2}, \eqref{I6} in $(0,T) \times \Omega$ in the sense specified in Section \ref{Ssw1} (see \eqref{sw7},\eqref{sw8}) belonging to the class
\begin{equation} \label{ws2}
\begin{split}
\tvu &\in C([0,T] \times \Ov{\Omega}; R^d),\ \tvr \in C([0,T] \times {\Omega}), \ \tvr \geq 0, \\
\nabla^2_x \tvu &\in L^2(0,T; L^q(\Omega; R^{d \times d \times d})),\
\partial_t \tvu \in L^2(0,T; L^q(\Omega; R^d) ),
\end{split}
\end{equation}
where
\begin{equation} \label{ws3}
1 < \gamma \leq 2,\ q \geq \frac{2 \gamma}{\gamma - 1}.
\end{equation}
Let $(\vr, \vu)$ be a weak solution of the same problem in the sense of Definition \ref{Dsw1} such that
\[
\vr(0, \cdot) = \tvr(0, \cdot) = \vr_0, \ (\vr \vu)(0, \cdot) = (\tvr \tvu)(0, \cdot) = \vm_0, 
\vu|_{\partial \Omega} = \tvu|_{\partial \Omega} = \vuB,
\]
where
\begin{equation} \label{ws4}
(\vr_0)^{\gamma - 1} \in W^{1,q}(\Omega; R^d).
\end{equation}

Then
\[
\vr = \tvr, \ \vu = \tvu \ \mbox{in}\ (0,T) \times \Omega.
\]

\end{Theorem}

\begin{Remark}\label{wsR1}

{In the Newtonian case, the local in time existence of strong solutions with vacuum in the case $\vu_B=0$ (no in/outflow) on smooth bounded domains, say $\Omega\in C^3$, with the pressure law (\ref{sw1a}) was proved
in Cho, Kim \cite[Theorem 3]{ChoKim}. The solutions belong to a regularity class included in \eqref{ws1} if $q = 6$.
Local existence for general non-Newtonian viscous stress was recently established by Kalousek, M\' acha, and Ne\v casov\' a 
\cite{KaMaNe} on periodic spatial domains.}

\end{Remark}

\begin{proof}

The strategy is to apply the relative energy inequality \eqref{re1} to $(\tvr, \tvu)$. This cannot be done in a direct manner as the latter requires $\inf \tvr > 0$. Thus we start with the choice $(\tvr + \ep, \tvu)$, $\ep > 0$, obtaining
\[
\begin{split}
&\intO{E\left( \vr, \vu \ \Big|\ \tvr + \ep , \tvu \right)(\tau, \cdot) } + \int_{\Ov{\Omega}} 1 \D \ \mathfrak{E} (\tau) \\
&+\int_0^\tau \intO{ \Big[ F(\Ds \vu) + F^* (\mathbb{S}) \Big] } \dt - \int_0^\tau \intO{ \mathbb{S} : \Grad \tvu } \dt
\\
&\leq \intO{ P(\vr_0) + \ep P'(\vr_0) - P(\vr_0 + \ep) }\\
&- \int_0^\tau \intO{ \vr (\tvu - \vu) \cdot (\tvu - \vu) \cdot \Grad \tvu } \dt\\
&-
\int_0^\tau \intO{ \Big[ p(\vr) - p'(\tvr + \ep) (\vr - \tvr - \ep) - p(\tvr + \ep) \Big] \Div \tvu } \dt
\\
&+ \int_0^\tau \intO{ \vr (\tvu - \vu) \cdot \Big[ \partial_t \tvu  +  \tvu \cdot \Grad \tvu
  + \Grad P'(\tvr + \ep) \Big] } \dt\\ &+ \ep \int_0^\tau \intO{
	p'(\tvr + \ep)\left( 1 - \frac{\vr}{\tvr + \ep} \right)  \Div \tvu
   } \dt
- \int_0^\tau \int_{\Ov{\Omega}} \Grad \tvu : \D \ \mathfrak{R}(t) \ \dt
\ \mbox{for a.a.}\ \tau \in (0,T),
\end{split}
\]
Moreover, as $\tvu$ belongs to the class \eqref{ws2} and $q \geq 4$,
we get $\| \Grad \tvu \|_{L^\infty(\Omega; R^{d \times d})} \in L^2(0,T)$, and
the above inequality reduces to
\begin{equation} \label{ws6}
\begin{split}
&\intO{E\left( \vr, \vu \ \Big|\ \tvr + \ep , \tvu \right)(\tau, \cdot) } + \int_{\Ov{\Omega}} 1 \D \ \mathfrak{E} (\tau) \\
&+\int_0^\tau \intO{ \Big[ F(\Ds \vu) + F^* (\mathbb{S}) \Big] } \dt - \int_0^\tau \intO{ \mathbb{S} : \Grad \tvu } \dt \\
&\leq \intO{ P(\vr_0) + \ep P'(\vr_0) - P(\vr_0 + \ep) } + \int_0^\tau \chi(t) \left[ \intO{E\left( \vr, \vu \ \Big|\ \tvr + \ep , \tvu \right)(t, \cdot) } + \int_{\Ov{\Omega}} 1 \D \ \mathfrak{E} (t) \right]
\dt  \\
 &+ \int_0^\tau \intO{ \vr (\tvu - \vu) \cdot \Big[ \partial_t \tvu  +  \tvu \cdot \Grad \tvu
  + \Grad P' (\tvr + \ep) \Big] } \dt\\ &+ \ep \int_0^\tau \intO{
	p'(\tvr + \ep)\left( 1 - \frac{\vr}{\tvr + \ep} \right) \Div \tvu } \dt, \ \mbox{with}\ \chi \in L^2(0,T).
\end{split}
\end{equation}

Next, we let $\ep \to 0$ in \eqref{ws6}. Obviously,
\[
\intO{E\left( \vr, \vu \ \Big|\ \tvr + \ep , \tvu \right)(\tau, \cdot) } \to
\intO{E\left( \vr, \vu \ \Big|\ \tvr , \tvu \right)(\tau, \cdot) }
 \ \mbox{as}\ \ep \to 0
\ \mbox{uniformly in}\ \tau \in (0,T),
\]
\[
\intO{ P(\vr_0) + \ep P'(\vr_0) - P(\vr_0 + \ep) } \to 0 \ \mbox{as}\ \ep \to 0,
\]
and
\[
\ep \int_0^\tau \intO{ p'(\tvr + \ep) \Div \tvu } \to 0 \ \mbox{as}\ \ep \to 0.
\]
In addition,
\[
\ep \vr \frac{p'(\tvr + \ep) }{\tvr +\ep} \aleq \ep \vr (\tvr + \ep)^{\gamma - 2} \leq
\ep^{\gamma - 1} \vr \to 0 \ \mbox{in}\ L^\gamma(\Omega) \ \mbox{as}\ \ep \to 0.
\]
Finally, by virtue of Lemma \ref{Lsw1} and hypotheses \eqref{ws2}, \eqref{ws4}, we have
\begin{equation} \label{ws7bis}
\Grad (\tvr)^{\gamma - 1} \in L^\infty(0,T; L^q(\Omega; R^d)).
\end{equation}
Seeing that
\[
\vr \vu \in L^\infty(0,T; L^{\frac{2 \gamma}{\gamma + 1}}(\Omega; R^3))
\]
we may use hypothesis \eqref{ws3} to conclude
\[
\int_0^\tau \intO{ \vr (\tvu - \vu) \cdot
\Grad P' (\tvr + \ep) } \dt \to
\int_0^\tau \intO{ \vr (\tvu - \vu) \cdot
\Grad P' (\tvr) } \dt \ \mbox{as}\ \ep \to 0.
\]
Performing the limit $\ep \to 0$ in \eqref{ws2} we may infer that
\begin{equation} \label{ws7}
\begin{split}
&\intO{E\left( \vr, \vu \ \Big|\ \tvr , \tvu \right)(\tau, \cdot) } + \int_{\Ov{\Omega}} 1 \D \ \mathfrak{E} (\tau) \\
&+\int_0^\tau \intO{ \Big[ F(\Ds \vu) + F^* (\mathbb{S}) \Big] } \dt - \int_0^\tau \intO{ \mathbb{S} : \Grad \tvu } \dt \\
&\leq \int_0^\tau \chi(t) \left[ \intO{E\left( \vr, \vu \ \Big|\ \tvr , \tvu \right)(t, \cdot) } + \int_{\Ov{\Omega}} 1 \D \ \mathfrak{E} (t) \right]
\dt  \\
 &+ \int_0^\tau \intO{ \vr (\tvu - \vu) \cdot \Big[ \partial_t \tvu  +  \tvu \cdot \Grad \tvu
  + \Grad P' (\tvr) \Big] } \dt.
\end{split}
\end{equation}

Now, as $(\tvr, \tvu)$ is a strong solution, we have
\begin{equation} \label{ws8}
\tvr \partial_t \tvu  +  \tvr \tvu \cdot \Grad \tvu
  + \tvr \Grad P' (\tvr) = \Div \widetilde{\mathbb{S}}.
\end{equation}
Consequently, we may add the integral
\[
\int_0^\tau \intO{ (\vu - \tvu) \cdot \Div \widetilde{\mathbb{S} }} \dt =
\int_0^\tau \intO{ \Ds (\tvu - \vu) : \widetilde{\mathbb{S} }} \dt
\]
to both sides of the inequality \eqref{ws7}. Regrouping terms on the left--hand side we get
\[
\begin{split}
\int_0^\tau & \intO{ \left[ F(\Ds \vu) + F^*(\mathbb{S}) - \mathbb{S} : \Ds \tvu +
\widetilde{\mathbb{S}} : (\Ds \tvu - \Ds \vu ) \right] } \dt\\
&= \int_0^\tau  \intO{ \left[ F(\Ds \vu) - F(\Ds \tvu) -
\widetilde{\mathbb{S}} : (\Ds \vu - \Ds \tvu ) \right] },
\end{split}
\]
where we have used Fenchel--Young inequality
\[
F(\Ds \tvu) + F^*(\mathbb{S}) \geq  \mathbb{S} : \Ds \tvu.
\]
As $\widetilde{\mathbb{S}} \in \partial F(\Ds \tvu)$ we can apply hypothesis \eqref{ws1}
for $\mathbb{D} = \Ds \tvu$, $\mathbb{S} = \widetilde{\mathbb{S}}$, $\mathbb{Q} = \Ds \vu - \Ds \tvu$
to obtain
\[
\begin{split}
\int_0^\tau &\intO{ \left| (\Ds \vu - \Ds \tvu) - \frac{1}{d} {\rm tr}[(\Ds \vu - \Ds \tvu)] \mathbb{I} \right|^2 }
\\ &\aleq
\int_0^\tau  \intO{ \left[ F(\Ds \vu) + F^*(\mathbb{S}) - \mathbb{S} : \Ds \tvu +
\widetilde{\mathbb{S}} : (\Ds \tvu - \Ds \vu ) \right] } \dt.
\end{split}
\]
Finally, using the traceless version of Korn's inequality, we may rewrite \eqref{ws7} in the form
\begin{equation} \label{ws9}
\begin{split}
\intO{E\left( \vr, \vu \ \Big|\ \tvr , \tvu \right)(\tau, \cdot) } &+ \int_{\Ov{\Omega}} 1 \D \ \mathfrak{E} (\tau) +
\int_0^\tau \intO{ | \Grad \vu - \Grad \tvu |^2 } \dt \\
&\aleq \int_0^\tau \chi(t) \left[ \intO{E\left( \vr, \vu \ \Big|\ \tvr , \tvu \right)(t, \cdot) } + \int_{\Ov{\Omega}} 1 \D \ \mathfrak{E} (t) \right]
\dt  \\
 &+ \int_0^\tau \intO{ (\vr - \tvr) (\tvu - \vu) \cdot \Big[ \partial_t \tvu  +  \tvu \cdot \Grad \tvu
  + \Grad P' (\tvr) \Big] } \dt.
\end{split}
\end{equation}

Thus it remains to control the last integral in \eqref{ws9}. Let $\Ov{\vr}$ be a positive constant
chosen so that
$\tvr \leq \frac{1}{2} \Ov{\vr}$ in $(0,T) \times \Omega$. We consider two complementary cases. Suppose first that $\vr \geq \Ov{\vr}$.
By virtue of hypotheses \eqref{ws2} -- \eqref{ws4}, we get
\begin{equation} \label{ws10}
\begin{split}
\int_{\Omega} & 1_{\vr \geq \Ov{\vr}} (\vr - \tvr) (\tvu - \vu) \cdot \Big[ \partial_t \tvu  +  \tvu \cdot \Grad \tvu
  + \Grad P' (\tvr) \Big] \dx\\
&\leq \left\| 1_{\vr \geq \Ov{\vr} } (\vr - \tvr)^{1/2} \right\|_{L^{2 \gamma}(\Omega)}
\left\| 1_{\vr \geq \Ov\vr} (\vr - \tvr)^{1/2} (\vu - \tvu) \right\|_{L^2(\Omega; R^d)}
\left\| \partial_t \tvu  +  \tvu \cdot \Grad \tvu
  + \Grad P' (\tvr) \right\|_{L^q(\Omega; R^d)} \\
&\aleq \left\| \partial_t \tvu  +  \tvu \cdot \Grad \tvu
  + \Grad P' (\tvr) \right\|_{L^q(\Omega; R^d)} \intO{ E \left( \vr, \vu \Big| \tvr , \tvu \right) }.	
\end{split}
\end{equation}

If $0 \leq \vr \leq \Ov{\vr}$, then
\begin{equation} \label{ws11}
\begin{split}
\int_{\Omega} &1_{\vr \leq \Ov{\vr}}(\vr - \tvr) (\tvu - \vu) \cdot \Big[ \partial_t \tvu  +  \tvu \cdot \Grad \tvu
  + \Grad P' (\tvr) \Big] \dx\\
&\leq \left\| 1_{\vr \leq \Ov{\vr} } (\vr - \tvr) \right\|_{L^{2}(\Omega)}
\left\| \vu - \tvu \right\|_{L^{2 \gamma} (\Omega; R^d)}
\left\| \partial_t \tvu  +  \tvu \cdot \Grad \tvu
  + \Grad P' (\tvr) \right\|_{L^q(\Omega; R^d)} \\
&\leq \delta \| \vu - \tvu \|^2_{L^{2 \gamma}(\Omega; R^d)}
+
c(\delta) \left\| 1_{\vr \leq \Ov{\vr} } (\vr - \tvr) \right\|_{L^{2}(\Omega)}^2 \left\| \partial_t \tvu  +  \tvu \cdot \Grad \tvu
  + \Grad P' (\tvr) \right\|_{L^q(\Omega; R^d)}^2  .	
\end{split}
\end{equation}
for any $\delta > 0$.
As $\gamma \leq 2$, the standard Poincar\' e--Sobolev inequality yields
\begin{equation} \label{ws12}
\delta \| \vu - \tvu \|^2_{L^{2 \gamma}(\Omega; R^d)} \leq \frac{1}{2} \intO{ |\Grad \vu - \Grad \tvu|^2 }
\end{equation}
for a suitably small $\delta > 0$. By the same token $P$ is strictly convex on the compact interval $[0, \Ov{\vr}]$; whence
\begin{equation} \label{ws13}
\left\| 1_{\vr \geq \Ov{\vr} } (\vr - \tvr) \right\|_{L^{2}(\Omega)}^2 \aleq \intO{
E \left( \vr, \vu \Big| \tvr, \tvu \right) }.
\end{equation}

In view of hypothesis \eqref{ws2}, and \eqref{ws7bis}, we have
\[
\left\| \partial_t \tvu  +  \tvu \cdot \Grad \tvu
  + \Grad P' (\tvr) \right\|_{L^q(\Omega; R^d)} \in L^2(0,T).
\]
Thus
plugging \eqref{ws10}--\eqref{ws13} in \eqref{ws9} we may use the standard Gronwall argument to conclude
\[
\intO{ E \left( \vr, \vu \Big| \tvr, \tvu \right)(\tau, \cdot) } = 0,
\ \tau \in (0,T) ,\ \int_0^T \intO{ |\Grad \vu - \Grad \tvu |^2 } \dt = 0.
\]

\end{proof}

The extension of the above result to more general viscous potentials still satisfying the coercivity condition \eqref{ws1} is straightforward, see \cite{AbbFeiNov}. The major drawback of Theorem \ref{wsT1} is the absence of inflow. 
If $\Gamma_{\rm in} \ne \emptyset$, a slight modification of the above arguments yields a positive result provided the boundary density 
$\vrB$ is bounded below away from vacuum. 
 
\begin{Theorem} [{\bf Weak--strong uniqueness, bounded domain, int/out flow}] \label{wsT1bis} 
Suppose that $\Omega \subset R^d$, $d=1,2,3$ is a bounded Lipschitz domain. Let the boundary conditions be determined by 
$\vrB \in C^1_c(R^d)$, $\vuB \in C^2_c (R^d; R^d)$, where 
\begin{equation} \label{novac}
\vrB(x) \geq \underline{\vr} > 0 \ \mbox{for any}\ x \in \Gamma_{\rm in}.
\end{equation} 
 Let $(\tvr, \tvu)$ be a strong solution of the problem \eqref{I1}, \eqref{I2}, \eqref{I6} in $(0,T) \times \Omega$ in the sense specified in Section \ref{Ssw1} belonging to the class
\begin{equation} \label{ws2bis}
\begin{split}
\tvu &\in C([0,T] \times \Ov{\Omega}; R^d),\ \tvr \in C([0,T]; W^{1,q}(\Omega)), \ \tvr \geq 0, \\
\nabla^2_x \tvu &\in L^2(0,T; L^q(\Omega; R^{d \times d \times d})),\
\partial_t \tvu \in L^2(0,T; L^q(\Omega; R^d) ),
\end{split}
\end{equation}
where
\[
1 < \gamma \leq 2,\ q \geq \frac{2 \gamma}{\gamma - 1}.
\]
Let $(\vr, \vu)$ be a weak solution of the same problem in the sense of Definition \ref{Dsw1} such that
\[
\begin{split}
\vr(0, \cdot) &= \tvr(0, \cdot) = \vr_0, \ (\vr \vu)(0, \cdot) = (\tvr \tvu)(0, \cdot) = \vm_0,\\ 
\vu|_{\partial \Omega} &= \tvu|_{\partial \Omega} = \vuB,\ 
(\vr \vu) \cdot \vc{n}|_{\Gamma_{\rm in}} = 
(\tvr \tvu) \cdot \vc{n}|_{\Gamma{\rm in}} = \vrB \vuB \cdot \vc{n},
\end{split}
\]
where
\begin{equation} \label{ws4bis}
(\vr_0)^{\gamma - 1} \in W^{1,q}(\Omega; R^d).
\end{equation}

Then
\[
\vr = \tvr, \ \vu = \tvu \ \mbox{in}\ (0,T) \times \Omega.
\]

\end{Theorem}

\begin{Remark} \label{vacR}

The class \eqref{ws2bis} is slightly smaller than \eqref{ws2} in Theorem \ref{wsT1} but still large enough to accommodate the strong solutions obtained in the Newtonian non--degenerate case by Valli and Zajaczkowski \cite{VAZA}. Note, however, that quite severe restrictions are imposed on $\Gamma_{\rm in}$ in \cite{VAZA}.

\end{Remark}

\begin{proof}

The proof can be done following the same arguments as in Theorem \ref{wsT1} as soon as we observe that 
\[
\tvr^{\gamma - 1} \in L^\infty(0,T; W^{1,q}(\Omega)).
\]
To see this, write 
\[
\tvr^{\gamma - 1} = r_1 + r_2,\ r_1 = \delta + [\tvr^{\gamma - 1} - \delta]^+ , \ r_2 = [\tvr^{\gamma - 1} - \delta]^-  \ \mbox{for a suitable}\ \delta > 0,
\]
where $a^+=\max\{a,0\}$, $a^-=\min\{a,0\}$.
Consequently, it is enough to show 
\[
\Grad r_i \in L^\infty(0,T; L^q(\Omega; R^d)),\ i = 1,2.
\]

As for $r_1$, we have 
\[
\Grad r_1 =
\Grad [ \tvr^{\gamma - 1} - \delta ]^+ = (\gamma - 1) {\rm sgn}^+ [ \tvr^{\gamma - 1} - \delta ] \tvr^{\gamma - 2} 
\Grad \tvr,
\]
so the desired conclusion follows from the hypothesis \eqref{ws2bis} as long as $\delta > 0$.

To obtain a similar estimate for $r_2$, we first extend the velocity field $\tvu$ and the initial density 
$\vr_0$
outside $\Omega$ 
so that they satisfy the hypotheses \eqref{ws2bis}, \eqref{ws4bis} on $R^d$. 
Accordingly, the equation of continuity \eqref{I1}
endowed with the extended velocity field $\tvu$ admits a solution $r$ determined uniquely through formula \eqref{sw4}. By virtue of 
Lemma \ref{Lsw1}, we have 
\begin{equation} \label{nov2}
r^{\gamma - 1} \in L^\infty (0,T; W^{1,q}(R^d)). 
\end{equation}

We say that a point $(t,x) \in [0,T] \times \Omega$ is regular, if the value of $\tvr(t,x)$ is determined by formula \eqref{sw4}, otherwise 
$(t,x)$ is singular. Clearly the backward characteristic curve emanating from a singular point reaches the boundary $\partial \Omega$ at a positive time, and $\tvr(t,x)$ is then given by \eqref{sw5}. Obviously, 
\[
\tvr(t,x) = r(t,x) \ \mbox{whenever}\ (t,x) \ \mbox{is regular}.
\]

As stated in \eqref{novac}, $\vrB$ is bounded below away from zero. Consequently, there exists $\ep = \ep (T, \tvu) > 0$ such that 
\[
\tvr^{\gamma - 1}(t,x) \geq \ep > 0 \ \mbox{whenever}\ (t,x) \ \mbox{is singular}. 
\]
Consequently, for $0 < \delta < \frac{\ep}{2}$ and any fixed $t \in [0,T]$, the set $\{ x \in \Omega \ \Big|\ \tvr^{\gamma - 1} (t,x) \leq \delta \}$ admits 
and open neighbourhood in $\Omega$ that consists of regular points. Consequently, 
\[
\Grad r_2 = \Grad [\tvr^{\gamma - 1} - \delta]^- = (\gamma - 1){\rm sgn}^- [\tvr^{\gamma - 1} - \delta] \Grad 
r^{\gamma -1},
\]
and the desired conclusion follows from \eqref{nov2}.

\end{proof}

\section{Weak--strong uniqueness, unbounded domains}
\label{ud}

The main difficulty when extending the previous results to unbounded domains is the lack of Poincar\' e--Sobolev inequality. For this reason, we consider only the case $d=3$, where
Sobolev's inequality is available, namely
\begin{equation} \label{ud1}
\| v \|_{L^6(\Omega)} \aleq \| \Grad v \|_{L^2(\Omega,R^3)} \ \mbox{for any}\ v \in D^{1,2}_0(\Omega).
\end{equation}
For simplicity, we restrict ourselves to the case $\vuB = 0$.

\subsection{General exterior domain}
\label{ged}

We look for strong solutions $(\tvr, \tvu)$ belonging to the class introduced by Huang et al. \cite{HuLiXi}. In particular,
\begin{equation} \label{ud2}
\begin{split}
\tvu &\in C([0,T]; D^{1,2}_0 \cap D^{3,2}(\Omega; R^3)) \cap L^2(0,T; D^{4,2}(\Omega; R^3)),\\
\partial_t \tvu &\in L^\infty(0,T; D^{1,2}_0 (\Omega, R^3)) \cap L^2(0,T; D^{2,2}(\Omega; R^3)), \\
\tvr,\ p(\tvr) &\in C([0,T]; W^{3,2}(\Omega)).
\end{split}
\end{equation}
We point out that the class is pertinent to the Newtonian case, where the momentum equation can be 
differentiated with respect to time to obtain higher order estimates on $\partial_t \tvu$.

In addition, similarly to Theorem \ref{wsT1}, we impose the condition
\begin{equation} \label{ud3}
(\vr_0)^{\gamma - 1} \in W^{1,6} \cap W^{1, \infty}(\Omega; R^3).
\end{equation}
As shown in Lemma \ref{Lsw1}, this regularity will propagate in time; whence we get
\begin{equation} \label{ud4}
\sup_{t \in (0,T)} \| \Grad (\tvr)^{\gamma - 1}(t, \cdot) \|_{L^6 \cap L^\infty (\Omega; R^3)} \leq c.
\end{equation}

Now, following step by step the proof of Theorem \ref{wsT1} we arrive at the inequality \eqref{ws9}, specifically,
\begin{equation} \label{ud5}
\begin{split}
\int_{\Omega} E\left( \vr, \vu \ \Big|\ \tvr , \tvu \right)(\tau, \cdot) \dx &+ \int_{\Omega} 1 \D \ \mathfrak{E} (\tau) +
\int_0^\tau \int_{\Omega}{ | \Grad \vu - \Grad \tvu |^2 }\dx \dt \\
&\aleq \int_0^\tau \chi(t) \left[ \int_{\Omega}{E\left( \vr, \vu \ \Big|\ \tvr , \tvu \right)(t, \cdot) } \dx + \int_{\Ov{\Omega}} 1 \D \ \mathfrak{E} (t) \right]
\dt  \\
 &+ \int_0^\tau \int_{\Omega}{ (\vr - \tvr) (\tvu - \vu) \cdot \vc{b} }\dx \dt,
\end{split}
\end{equation}
where, by virtue of \eqref{ud2}, \eqref{ud4},
\begin{equation} \label{ud5a}
\vc{b} = \partial_t \tvu + \tvu \cdot \Grad \tvu + \Grad P'(\tvr) \in L^2(0,T; L^6 \cap L^\infty (\Omega; R^3)).
\end{equation}
Similarly to the preceding section, we write
\[
\begin{split}
\int_0^\tau &\int_{\Omega}{ (\vr - \tvr) (\tvu - \vu) \cdot \vc{b} }\dx \dt \\ &=
\int_0^\tau \int_{\Omega}{ 1_{\vr \geq \Ov{\vr}} (\vr - \tvr) (\tvu - \vu) \cdot \vc{b} }\dx \dt
+ \int_0^\tau \int_{\Omega}{ 1_{\vr < \Ov{\vr}} (\vr - \tvr) (\tvu - \vu) \cdot \vc{b} }\dx \dt,
\end{split}
\]
where, exactly as in \eqref{ws10},
\begin{equation} \label{ud6}
\begin{split}
\int_0^\tau &\int_{\Omega}{ 1_{\vr \geq \Ov{\vr}} (\vr - \tvr) (\tvu - \vu) \cdot \vc{b} }\dx \dt \\
&\leq \left\| 1_{\vr \geq \Ov{\vr} } (\vr - \tvr)^{1/2} \right\|_{L^{2}(\Omega)}
\left\| 1_{\vr \geq \Ov\vr} (\vr - \tvr)^{1/2} (\vu - \tvu) \right\|_{L^2(\Omega; R^3)}
\left\| \vc{b} \right\|_{L^\infty(\Omega; R^3)} \\
&\aleq \left\| \vc{b} \right\|_{L^\infty(\Omega; R^3)} \intO{ E \left( \vr, \vu \Big| \tvr , \tvu \right) }.	
\end{split}
\end{equation}

The integral
\begin{equation} \label{ud7}
\int_0^\tau \int_{\Omega}{ 1_{\vr < \Ov{\vr}} (\vr - \tvr) (\tvu - \vu) \cdot \vc{b} }\dx \dt
\end{equation}
is more difficult to handle. Indeed, in order to mimick the arguments leading to \eqref{ws11}, we would need
\[
\vc{b} \ \mbox{bounded in}\ L^2(0,T; L^3(\Omega; R^3)),
\]
which does not follow from \eqref{ud2} if $\Omega$ is an exterior domain and must be imposed as an extra hypothesis.

\begin{Theorem} [{\bf Weak--strong uniqueness, unbounded domains}] \label{wsT3}

Suppose that $\Omega \subset R^3$, $d=1,2,3$ is an exterior domain with compact Lipschitz boundary.
Let $(\tvr, \tvu)$ be a strong solution of the problem \eqref{I1}, \eqref{I2}, \eqref{I6} with $\vuB = 0$ in $(0,T) \times \Omega$ in the sense specified in Section \ref{Ssw1} belonging to the class \eqref{ud2}. In addition, suppose that
\begin{equation} \label{HYP}
\partial_t \tvu \in L^2(0,T; L^3(\Omega; R^3)).
\end{equation}

Let $(\vr, \vu)$ be a weak solution of the same problem in the sense of Definition \ref{Dsw1} such that
\[
\vr(0, \cdot) = \tvr(0, \cdot) = \vr_0, \ (\vr \vu)(0, \cdot) = (\tvr \tvu)(0, \cdot) = \vm_0,
\]
where
\begin{equation} \label{HYP1}
(\vr_0)^{\gamma - 1} \in W^{1,3} \cap W^{1, \infty}(\Omega; R^3).
\end{equation}

Then
\[
\vr = \tvr, \ \vu = \tvu \ \mbox{in}\ (0,T) \times \Omega.
\]

\end{Theorem}

The result extends easily to general in/out flow boundary conditions exactly as in Theorem \ref{wsT1bis}.

\subsection{Compactly supported density}

In view of the existence class \eqref{ud2} identified in \cite{HuLiXi}, the hypothesis \eqref{HYP} does not seem very realistic, 
in particular in the presence of vacuum. To handle the general case, we restrict ourselves to the initial data with \emph{compactly supported}
initial density $\vr_0$ considered in \cite{HuLiXi} or \cite{XIN}.
Moreover, we restrict ourselves to the class of Newtonian fluids, where 
\[ 
\mathbb{S}(\Grad \tvu) = \mu \left( \Grad \tvu + \Grad^t \tvu - \frac{2}{3} \Div \tvu \mathbb{I} \right) + \lambda \Div \tvu \mathbb{I},\ \mu > 0,\ \lambda \geq 0.
\]

In view of Lemma \ref{Lsw1}, there exists $R > 0$ sufficiently large such that
\begin{equation} \label{ud8}
\tvr(t,x) = 0 \ \mbox{for all}\ t \in [0,T], \ |x| \geq R.
\end{equation}
Going back to the integral \eqref{ud7}, we get
\[
\begin{split}
\int_0^\tau &\int_{\Omega}{ 1_{\vr < \Ov{\vr}} (\vr - \tvr) (\tvu - \vu) \cdot \vc{b} }\dx \dt \\
&=
\int_0^\tau \int_{x \in \Omega, |x| \leq R}{ 1_{\vr < \Ov{\vr}} (\vr - \tvr) (\tvu - \vu) \cdot \vc{b} }\dx \dt +
\int_0^\tau \int_{|x| > R }{ 1_{\vr < \Ov{\vr}} \vr(\tvu - \vu) \cdot \vc{b} }\dx \dt,
\end{split}
\]
where
\[
\int_{x \in \Omega, |x| \leq R}{ 1_{\vr < \Ov{\vr}} (\vr - \tvr) (\tvu - \vu) \cdot \vc{b} }\dx \leq \left\| 1_{\vr \leq \Ov{\vr}} (\vr - \tvr) \right\|_{L^2(\Omega)}
\| \vu - \tvu \|_{L^6 (\Omega; R^3)}  \| \vc{b} \|_{L^3(\Omega \cap \{ |x| \leq R \})}.
\]
Seeing that
\[
\| \vc{b} \|_{L^3(\Omega \cap \{ |x| \leq R; R^3 \})} \leq c(R) \| \vc{b} \|_{L^6(\Omega; R^3)}
\]
the above integral is controlled by \eqref{ud5a}.

Finally,
\begin{equation} \label{ud10a}
\int_{|x| > R }{ 1_{\vr < \Ov{\vr}} \vr(\tvu - \vu) \cdot \vc{b} }\dx \leq \delta \| \vu - \tvu \|^2_{L^6(\Omega; R^3)} +
c(\delta) \| 1_{\vr \leq \Ov{\vr} } \vr \|^2_{L^\gamma(\Omega)} \| \vc{b} \|_{L^q(|x| \geq R) }^2
\end{equation}
for any $\delta > 0$, where
\[
\frac{1}{q} = \frac{5}{6} - \frac{1}{\gamma}.
\]
Suppose that $1 < \gamma \leq 2$. Consequently,
\[
\| 1_{\vr \leq \Ov{\vr} } \vr \|^2_{L^\gamma(\Omega)} \leq c(\Ov{\vr}) \| 1_{\vr \leq \Ov{\vr} } \vr \|^\gamma_{L^\gamma(\Omega)} \aleq
\int_{\Omega} E \left( \vr, \vu \ \Big| \tvr, \tvu \right) \dx.
\]
Thus, to close the estimates, we need to control
\begin{equation} \label{ud10}
\int_0^T
\| \vc{b} \|_{L^q(|x| \geq R)}^2 \dt \leq c(q) \ \mbox{for any}\ q > 3.
\end{equation}
To see \eqref{ud10}, we realize that
\[
\vc{b} = \partial_t \tvu + \tvu \cdot \Grad\tvu \ \mbox{if}\ |x| > R.
\]
If the fluid is Newtonian, we get immediately from \eqref{ws8} that
\[
\mu
\Div \left( \Grad \partial_t \tvu + \Grad^t \partial_t \tvu - \frac{2}{3} \Div \partial_t \tvu \mathbb{I} \right)
+ \lambda \Grad \Div \partial_t \tvu = 0
\ \mbox{for}\ |x| > R,
\]
where, in view of \eqref{ud2}
\[
\partial_t \tvu \in D^{1,2} \cap L^6(|x| > R), \ \partial_t \tvu \in C^\alpha (|x| = R)
\ \mbox{for some}\ \alpha > 0 \ \mbox{and a.a.} \ t \in (0,T).
\]
Using the standard elliptic estimates, we get
\[
|\partial_t \tvu (t,x) | \aleq \frac{1}{|x|} \|\partial_t \tvu (t, \cdot)\|_{C(|x| = R)} \ \mbox{for all} \ |x| \geq R,
\]
which, together with \eqref{ud2}, yields \eqref{ud10}. Thus we are able to control the integral \eqref{ud10a} as soon as 
$1 < \gamma < 2$.

We have shown the following result.

\begin{Theorem} [{\bf Weak--strong uniqueness, compactly supported density}] \label{wsT2}

Suppose that $\Omega \subset R^3$ is a Lipschitz exterior domain, $\vuB = 0$, and $1 < \gamma < 2$. In addition, let $\mathbb{S}$ be Newtonian,
\begin{equation} \label{ud10bis}
\mathbb{S} = \mu \left( \Grad \vu + \Grad^t \vu - \frac{2}{3} \Div \vu \mathbb{I} \right) + \lambda \Div \vu \mathbb{I},\ \mu > 0,\ \lambda \geq 0.
\end{equation}
Let $(\tvr, \tvu)$ be a strong solution of the problem \eqref{I1}, \eqref{I2} in $(0,T) \times \Omega$ in the sense specified in Section \ref{Ssw1} belonging to the class
\eqref{ud2}.
Let $(\vr, \vu)$ be a weak solution of the same problem in the sense of Definition \ref{Dsw1} such that
\[
\vr(0, \cdot) = \tvr(0, \cdot) = \vr_0, \ (\vr \vu)(0, \cdot) = (\tvr \tvu)(0, \cdot) = \vm_0,
\]
where
\[
\vr_0 \in C_c(\Ov{\Omega}),\
(\vr_0)^{\gamma - 1} \in W^{1,6} \cap W^{1, \infty}(\Omega; R^3).
\]

Then
\[
\vr = \tvr, \ \vu = \tvu \ \mbox{in}\ (0,T) \times \Omega.
\]

\end{Theorem}

\begin{Remark}\label{udR1}

Local existence of strong solutions as well as global existence for small initial data in the class included in (\ref{ud2}) was proved by  Huang et al. \cite[Lemma 2.1]{HuLiXi}, see also Cho and Kim \cite[Theorem 3]{ChoKim} for local existence in the class of 
more regular solution.

\end{Remark}

The result can be extended to the case of general in/out flow boundary conditions in the spirit of Theorem \ref{wsT1bis}. The hypothesis of Newtonian viscous stress is however necessary.

\subsection{Positive density}

Finally, we consider the physically relevant case $\Omega = R^3$,
\begin{equation} \label{ud13}
\begin{split}
0 < \vr_0 (x) &\leq \Ov{\vr}, \\  |\Grad (\vr_0)^{\gamma - 1}(x) | &\aleq \frac{1}{1 + |x|^\alpha}, \ \alpha = \alpha(\gamma) > 0,
\end{split}
\end{equation}
meaning vacuum is not present but $\vr_0$ decays to zero as
$|x| \to \infty$. We consider the Newtonian viscous stress \eqref{ud10bis} and restrict slightly the class \eqref{ud2} of strong solutions
\begin{equation} \label{ud15}
\begin{split}
\tvu &\in C([0,T]; D^{1,2}_0 \cap D^{3,2}(R^3; R^3)) \cap L^2(0,T; D^{4,2}(R^3; R^3)),\\
\partial_t \tvu &\in L^\infty(0,T; D^{1,2}_0 (R^3, R^3)) \cap L^2(0,T; D^{2,2}(R^3; R^3)), \\
\tvr,\ p(\tvr) &\in C([0,T]; W^{3,2}(R^3)), \\
\sqrt{\vr} \partial^2_{t,t} \tvu &\in L^2(0,T; L^2(R^3; R^3)).
\end{split}
\end{equation}
introduced by Cho and Kim \cite{ChoKim}.

The first observation is that any solution belonging to \eqref{ud15} satisfies
\[
\tvu \in C([0,T]; BC^1(R^3; R^3)), \ \nabla^2 \tvu \in L^2(0,T; BC(R^3; R^9)).
\]
Thus a direct inspection of formula \eqref{sw4} reveals that the decay properties of the initial density propagate in time.
More specifically, we have
\[
\begin{split}
| \Grad (\vr_0)^{\gamma - 1}(x) | \aleq \frac{1}{1 + |x|^\alpha} \ &\Rightarrow \
(\vr_0)^{\gamma - 1}(x) \aleq \frac{1}{1 + |x|^{\alpha - 1}}
\ \Rightarrow \vr^{\gamma - 1}(t,x) \aleq \frac{1}{1 + |x|^{\alpha - 1}} \\
&\Rightarrow \ |\Grad (\vr(t,x))^{\gamma - 1} | \aleq \frac{1}{1 + |x|^{\alpha - 1}}
\end{split}
\]
In particular, we get
\begin{equation} \label{ud17}
\Grad P'(\tvr) \in L^\infty(0,T; L^1 \cap L^\infty(R^3; R^3))
\ \mbox{provided}\ \alpha > 2,
\end{equation}
and
\begin{equation} \label{ud16}
\tvr(t,x) \approx \frac{1}{1 + |x|^\beta} ,\ |\Grad \tvr(t,x)| \aleq \frac{1}{1 + |x|^\beta} \ \mbox{for all}\ t \in[0,T],\ x \in R^3,\
\beta = \frac{\alpha - 1}{\gamma - 1}.
\end{equation}

Consequently, all steps in Section \ref{ged} up to formula \eqref{ud6} can be performed and we are left with the integral \eqref{ud7}
\begin{equation} \label{ud18}
\int_0^\tau \int_{R^3} 1_{\vr \leq \Ov{\vr}} (\vr - \tvr) (\vu - \tvu) \cdot \Big( \partial_t \tvu + \tvu \cdot \Grad \tvu + \Grad P'(\tvr) \Big)
\dx \dt.
\end{equation}

We proceed in several steps:

{\bf Step 1.}

As $\tvu \in L^\infty((0,T) \times R^3; R^3)$, $\Grad \tvu \in L^\infty \cap L^2((0,T) \times R^3; R^9)$, we get, by interpolation,
\[
\vu \cdot \Grad \tvu \in L^2(0,T; L^3(R^3, R^3))
\]
Similarly, it follows from \eqref{ud16} that
\begin{equation} \label{ud19}
\Grad P'(\tvr) \in L^\infty(0,T; L^3 \cap L^\infty(R^3; R^3)).
\end{equation}
Consequently,
\[
\begin{split}
\int_{R^3} &1_{\vr \leq \Ov{\vr} } (\vr - \tvr) (\vu - \tvu) \cdot \Big( \tvu \cdot \Grad \tvu + \Grad P'(\tvr) \Big)
\dx \dt \\ &\leq \| \vr - \tvr \|_{L^2(R^3)} \| \vu - \tvu \|_{L^6(R^3; R^3)} \| \tvu \cdot \Grad \tvu + \Grad P'(\tvr) \|_{L^3(R^3;R^3)}
\end{split},
\]
and we may use the arguments of Section \ref{ged} to control the expression on the right--side. Thus \eqref{ud18} reduces to
\begin{equation} \label{ud20}
\int_0^\tau \int_{R^3} 1_{\vr \leq \Ov{\vr}} (\vr - \tvr)(\vu - \tvu) \cdot \partial_t \tvu \dx \dt.
\end{equation}

{\bf Step 2.} To control $\partial_t \tvu$, we differentiate the momentum equation with respect to time obtaining
\begin{equation} \label{ud20bis}
\mathcal{L}[ \partial_t \tvu] = \partial^2_{t,t} (\tvr \tvu) + \Div ( \partial_t (\tvr \tvu \otimes \tvu ) ) +
\Grad \partial_t p(\tvr),
\end{equation}
where $\mathcal{L}$ denotes the Lam\' e elliptic operator
\begin{equation} \label{ud21}
\mathcal{L}[\vc{v}] = \mu \Div \left( \Grad \vc{v} + \Grad \vc{v}^t - \frac{2}{3} \Div \vc{v} \mathbb{I} \right) +
\lambda \Grad \Div \vc{v}.
\end{equation}
Now,
\[
\begin{split}
\partial_t (\tvr \tvu \otimes \tvu) &= \partial_t \tvr (\tvu \otimes \tvu) + \tvr (\partial_t \tvu \otimes \tvu) +
\tvr (\tvu \otimes \partial_t \tvu) \\ &= - \tvu \cdot \Grad \tvr (\tvu \otimes \tvu)
- \tvr \Div \tvu (\tvu \otimes \tvu) + \tvr (\partial_t \tvu \otimes \tvu) +
\tvr (\tvu \otimes \partial_t \tvu).
\end{split}
\]
By H\" older's inequality
\[
\| \tvu \cdot \Grad \tvr (\tvu \otimes \tvu) \|_{L^{\frac{3}{2}}(R^3; R^3) }
\leq \| \tvu \|_{L^6(R^3; R^3)}^3 \| \Grad \tvr \|_{L^6(R^3; R^3)}.
\]
Similarly
\[
\| \tvr \Div \tvu (\tvu \otimes \tvu) \|_{L^{\frac{3}{2}}(R^3; R^3) }  \leq
\| \tvr \|_{L^\infty(R^3)} \| \Div \tvu \|_{L^2 (R^3)} \| \tvu \|^2_{L^6(R^3)},
\]
and
\[
\| \tvr (\partial_t \tvu \otimes \tvu) +
\tvr (\tvu \otimes \partial_t \tvu) \|_{L^{\frac{3}{2}}(R^3; R^3) } \leq 2 \| \tvr \|_{L^3(R^3)}
\| \tvu \|_{L^6(R^3; R^3)} \| \partial_t \tvu \|_{L^6(R^3; R^3)}.
\]
As $(\tvr, \tvu)$ belongs to the class \eqref{ud15} and $\tvr$ satisfies \eqref{ud19}, it is easy to check that
\begin{equation} \label{ud22}
\partial_t (\tvr \tvu \otimes \tvu) \in L^2(0,T; L^{\frac{3}{2}}(R^3; R^9)).
\end{equation}

Next,
\[
\partial_t p(\tvr) = - p'(\tvr) \Big( \Grad \tvr \cdot \tvu - \tvr \Div \tvu \Big),
\]
where
\[
\| p'(\tvr) \Grad \tvr \cdot \tvu \|_{L^{\frac{3}{2}}(\Omega)} \aleq \| \tvr \|_{L^6(R^3)} \| \Grad P'(\tvr) \|_{L^3(R^3)} \| \tvu \|_{L^6(R^3)}
\]
Similarly to the above, we conclude
\begin{equation} \label{ud24}
\partial_t p(\tvr) \in L^\infty (0,T; L^{\frac{3}{2}}(R^3)).
\end{equation}
Going back to the equation \eqref{ud20bis}, we get
\[
\mathcal{L} [\partial_t \tvu] = \partial^{2}_{t,t} (\tvr \tvu) + \Div \mathbb{B}
\]
where
\[
\mathbb{B} \in L^2(0,T; L^{\frac{3}{2}}(R^3; R^9)).
\]
In view of the standard elliptic estimates, we may write
\[
\partial_t \tvu = \vc{z} + \vc{v},
\]
where
\begin{equation} \label{ud25}
\mathcal{L} [\vc{z}] = \partial^{2}_{t,t} (\tvr \tvu),
\end{equation}
while
\[
\vc{v} \in L^2(0,T; D^{1,\frac{3}{2}}_0(R^3; R^3)) \hookrightarrow L^2(0,T; L^3(R^3; R^3)).
\]
In view of the arguments employed in {\bf Step 1}, our task reduces to controlling the integral
\begin{equation} \label{ud26}
\int_0^\tau \int_{R^3} 1_{\vr \leq \Ov{\vr}} (\vr - \tvr)(\vu - \tvu) \cdot \vc{z} \dx \dt.
\end{equation}

{\bf Step 3.}

First, we write the right--hand side of the elliptic equation \eqref{ud25} in the form
\[
\partial^2_{t,t} (\tvr \tvu) = \partial^2_{t,t} \tvr \tvu + \tvr \partial^2_{t,t} \tvu - 2 \Div (\tvr \tvu) \partial_t \tvu
= - \partial_t \Div (\tvr \tvu) \tvu + \tvr  \partial^2_{t,t} \tvu  - 2
\tvr \Div(\tvu) \partial_t \tvu - 2 \Grad \tvr \cdot \tvu \partial_t \tvu
\]
where, furthermore,
\[
\partial_t \Div (\tvr \tvu) \tvu = \Div \left( \partial_t (\tvr \tvu)  \otimes \tvu \right) -
\partial_t (\tvr \tvu) \cdot \Grad \tvu.
\]
As the term $\Div \left( \partial_t (\tvr \tvu)  \otimes \tvu \right)$ can be handled exactly as in {\bf Step 2}, our task reduces to
estimating the integral
\begin{equation} \label{ud26bis}
\int_0^\tau \int_{R^3} 1_{\vr \leq \Ov{\vr}} (\vr - \tvr)(\vu - \tvu) \cdot \vc{w} \dx \dt,
\end{equation}
where $\vc{w}$ solves the elliptic system
\begin{equation} \label{ud25bis}
\mathcal{L}[\vc{w}] = \vc{g},
\end{equation}
with the right--hand side
\begin{equation} \label{ud28bis}
\vc{g} = \tvr \partial_t \tvu \cdot \Grad \tvu - \tvr \Div \tvu (\tvu \cdot \Grad \tvu) +
\tvu \Grad \tvr \cdot \tvu \cdot \Grad \tvu +
\tvr  \partial^2_{t,t} \tvu  - 2
\tvr \Div(\tvu) \partial_t \tvu - 2 \Grad \tvr \cdot \tvu \partial_t \tvu.
\end{equation}

Our goal is to show that that function $\vc{w}$ decays to zero for large $x$, more specifically, we shall see that
\begin{equation} \label{ud27}
|\vc{w}(t,x)| \leq \chi(t) \frac{1}{|x|} \ \mbox{with}\ \chi \in L^2(0,T).
\end{equation}
Taking \eqref{ud27} for granted, we use Hardy's inequality to estimate the integral \eqref{ud26bis},
\[
\begin{split}
\int_{R^3} 1_{\vr \leq \Ov{\vr}} (\vr - \tvr)(\vu - \tvu) \cdot \vc{w} \dx \leq
\delta \int_{R^3} \frac{ |\vu - \tvu|^2 }{|x|^2} \dx + c(\delta) \int_{R^3} \chi^2 1_{\vr \leq \Ov{\vr}} |\vr - \tvr|^2\\
\aleq \delta \int_{R^3} |\Grad \vu - \Grad \tvu|^2 \dx + c(\delta)
\int_{R^3} \chi^2 \mathcal{E} \left( \vr, \vu \Big| \tvr, \tvu \right), \ \delta > 0 \ \mbox{arbitrary.}
\end{split}
\]
Thus the proof of weak strong uniqueness can be completed via Gronwall's argument exactly as in Section \ref{wsbd}.

It remains to show \eqref{ud27}. As $\vc{w}$ solves the elliptic problem \eqref{ud25bis}, we get
\[
\vc{w}(t,x) = \int_{R^3} \mathbb{G}(x,y) \cdot \vc{g} (t, y) \D y,
\]
where $\mathbb{G}$ is the Green kernel associated to $\mathcal{L}$. Writing
\[
\begin{split}
\int_{R^3} \mathbb{G}(x,y) \cdot \vc{g} (t, y) \D y =
\frac{1}{|x|} \int_{R^3} (|x| - |y|) \mathbb{G}(x,y) \cdot \vc{g} (t, y) \D y \\ +
\frac{1}{|x|} \int_{R^3} |y| \mathbb{G}(x,y) \cdot \vc{g}  (t, y) \D y
\end{split}
\]
and using the fact
\[
|\mathbb{G}(x,y)| \aleq \frac{1}{|x - y|},
\]
we observe that \eqref{ud27} follows provided
\begin{equation} \label{ud28}
\| \vc{g}(t, \cdot) \|_{L^1(R^3; R^3)} + \| |x| \vc{g} (t, \cdot) \|_{L^q(R^3; R^3)}  \leq \chi(t) \ \mbox{for some}\
\ \chi \in L^2(0,T),\ q \in (\frac{3}{2} - \ep, \frac{3}{2} + \ep), \ep > 0.
\end{equation}

Thus our ultimate goal is to check \eqref{ud28} for all terms appearing on the right--hand side of \eqref{ud28bis}.
First,
we have
\[
\| \tvr \partial_t \tvu \cdot \Grad \tvu \|_{L^1(R^3)} \leq \| \tvr \|_{L^3(R^3)} \| \partial_t \tvu \|_{L^6(R^3)}
\| \Grad \tvu \|_{L^2(R^3)},
\]
and
\[
 \| |x| \tvr \partial_t \tvu \cdot \Grad \tvu \|_{L^q(R^3)} \leq \| \ |x| \tvr \ \|_{L^q(R^3)} \| \partial_t \tvu \|_{L^\infty(R^3)}
\| \Grad \tvu \|_{L^\infty(R^3)}.
\]
In view of \eqref{ud15}, the bound \eqref{ud28} holds as soon as $\beta > 3$. The same argument applies to
the terms $\tvr \Div \tvu \partial_t \tvu$, $\tvr \Div \tvu (\tvu \cdot \Grad \tvu)$. Moreover, seeing that $\Grad \tvr$ decays at least
as fast as $\tvr$, we can handle $\tvu \Grad \tvr \cdot \tvu \cdot \Grad \tvu$ and $\Grad \tvr \cdot \tvu \partial_t \tvu$ in the same manner. Here, we have used the fact that $\tvu$ is uniformly bounded.

Finally,
\[
\| \tvr \partial^2_{t,t} \tvu \|_{L^1(R^3)} \leq \| \sqrt{\tvr} \|_{L^2(R^3)} \| \sqrt{\tvr}  \partial^2_{t,t} \tvu \|_{L^2(R^3)},
\]
\[
\| |x| \tvr \partial^2_{t,t} \tvu \|_{L^q(R^3)} \leq \| |x| \sqrt{\tvr} \|_{L^p(R^3)} \| \sqrt{\tvr}  \partial^2_{t,t} \tvu \|_{L^2(R^3)}, \
\frac{1}{p} = \frac{1}{q} - \frac{1}{2}.
\]
Thus the choice $\beta > 3$ yields \eqref{ud28}.

We have shown the following result:

\begin{Theorem} [{\bf Weak--strong uniqueness, positive density}] \label{wsT5}

Suppose that $\Omega = R^3$ and $1 < \gamma \leq 2$. In addition, let $\mathbb{S}$ be Newtonian,
\[
\mathbb{S} = \mu \left( \Grad \vu + \Grad^t \vu - \frac{2}{3} \Div \vu \mathbb{I} \right) + \lambda \Div \vu \mathbb{I},\ \mu > 0,\ \lambda \geq 0.
\]
Let $(\tvr, \tvu)$ be a strong solution of the problem \eqref{I1}, \eqref{I2} in $(0,T) \times R^3$ in the sense specified in Section \ref{Ssw1} belonging to the class
\eqref{ud15}.
Let $(\vr, \vu)$ be a weak solution of the same problem in the sense of Definition \ref{Dsw1} such that
\[
\vr(0, \cdot) = \tvr(0, \cdot) = \vr_0, \ (\vr \vu)(0, \cdot) = (\tvr \tvu)(0, \cdot) = \vm_0,
\]
where
\[
\begin{split}
0 < \vr_0(x) &\leq \Ov{\vr}, \\
|\Grad (\vr_0)^{\gamma -1} (x) |  &\aleq \frac{1}{1 + |x|^\alpha},\ \alpha > \max \left\{2; 3\gamma - 2 \right\}.
\end{split}
\]

Then
\[
\vr = \tvr, \ \vu = \tvu \ \mbox{in}\ (0,T) \times R^3.
\]

\end{Theorem}

{
\begin{Remark}\label{udR1a}
 
Local existence of strong solutions in the class (\ref{ud15}) was proved by Cho and Kim \cite[Theorem 3]{ChoKim}, see also Huang et al. \cite[Lemma 2.1]{HuLiXi}.

\end{Remark}

Similarly to the preceding results, Theorem \eqref{wsT5} can be extended 
to exterior domains with general in/out flow boundary conditions. Vacuum can be accommodated at the expense of various technical 
difficulties.
Finally, it is worth noting that the condition $\alpha > 3 \gamma - 2$ in Theorem \ref{wsT5} is critical for the fluid to have finite mass, namely
\[
\int_{R^3} \tvr \dx < \infty.
\]

}

\def\cprime{$'$} \def\ocirc#1{\ifmmode\setbox0=\hbox{$#1$}\dimen0=\ht0
  \advance\dimen0 by1pt\rlap{\hbox to\wd0{\hss\raise\dimen0
  \hbox{\hskip.2em$\scriptscriptstyle\circ$}\hss}}#1\else {\accent"17 #1}\fi}


\end{document}